\documentclass[12pt,reqno,twoside]{amsart}

\usepackage{amssymb,amsmath,amscd,enumerate,verbatim}
\usepackage[colorlinks=true,linkcolor=blue,citecolor=blue]{hyperref}
\usepackage{amsmath}
\usepackage[latin1]{inputenc}
\usepackage{color}

\usepackage[all]{xy}
\usepackage{pstricks, pst-3d}

\newcommand{\mm}{\mathfrak m}
\newcommand{\nn}{\mathfrak n}

\newcommand{\Grsf}{\mathsf{Gr}}

\newcommand{\Z}{\mathbb{Z}}

\newcommand{\N}{\mathbb{N}}
\newcommand{\Q}{\mathbb{Q}}

\newcommand\bsx{{\boldsymbol x}}
\newcommand\bsy{{\boldsymbol y}}

\DeclareMathOperator{\pnt}{\raise 0.5mm \hbox{\large\bf.}}

\DeclareMathOperator{\depth}{depth}

\DeclareMathOperator{\Hom}{Hom}

\DeclareMathOperator{\indeg}{indeg}

\DeclareMathOperator{\Tor}{Tor}

\DeclareMathOperator{\Ext}{Ext}
\DeclareMathOperator{\exreg}{exreg}

\DeclareMathOperator{\Ann}{Ann}

\DeclareMathOperator{\grade}{grade}

\DeclareMathOperator{\Ker}{Ker}
\DeclareMathOperator{\kreg}{kreg}

\DeclareMathOperator{\injdim}{injdim}
\DeclareMathOperator{\img}{Im}
\DeclareMathOperator{\reg}{reg}

\DeclareMathOperator{\projdim}{pd}

\def\+#1{\relax\ifmmode\if\noexpand #1\relax \mathop{\kern
    0pt^+{#1}}\nolimits\else \kern 0pt^+\!#1 \fi\else$^*$#1\fi}

\newcommand{\dcat}[1]{{\mathsf D}(#1)}
\newcommand{\dcatgr}[1]{{\mathsf D}(\mathsf{Gr}\,#1)}
\newcommand{\dgrcat}[1]{{\mathsf D}^{\mathsf f}(\mathsf{Gr}\,#1)}
\newcommand{\dgrcatb}[1]{{\mathsf{D}^{\mathsf f}_{\mathsf b}}(\mathsf{Gr}\,#1)}
\newcommand{\dgrcatn}[1]{{\mathsf{D}^{\mathsf f}_{\!{\scriptscriptstyle\mathsf -}}}(\mathsf{Gr}\,#1)}
\newcommand{\dgrcatp}[1]{{\mathsf{D}^{\mathsf f}_{\!{\scriptscriptstyle\mathsf +}}}(\mathsf{Gr}\,#1)}
\newcommand{\dlc}{{\mathsf R}\Gamma}
\newcommand{\dtensor}[1]{\otimes^{\mathsf L}_{#1}}
\newcommand{\hgrcat}[1]{{\mathsf K}(\mathsf{Gr}\,#1)}
\newcommand{\hgrcatb}[1]{{\mathsf{K}^{\mathsf f}_{\mathsf b}}(\mathsf{Gr}\,#1)}
\newcommand{\hgrcatn}[1]{{\mathsf{K}^{\mathsf f}_{\!{\scriptscriptstyle\mathsf -}}}(\mathsf{Gr}\,#1)}
\newcommand{\hgrcatp}[1]{{\mathsf{K}^{\mathsf f}_{\!{\scriptscriptstyle\mathsf +}}}(\mathsf{Gr}\,#1)}
\newcommand{\RHom}[1]{{\mathsf{R}\!\Hom}_{#1}}
\newcommand{\grcat}[1]{{\mathsf{Gr}}\,#1}
\newcommand{\shift}{\mathsf{\Sigma}}

\let\:=\colon

\newtheorem{thm}{\bf Theorem}[section]
\newtheorem{lem}[thm]{\bf Lemma}
\newtheorem{cor}[thm]{\bf Corollary}
\newtheorem{prop}[thm]{\bf Proposition}

\theoremstyle{definition}
\newtheorem{defn}[thm]{\bf Definition}

\theoremstyle{plain}
\newtheorem*{thm*}{Theorem}

\theoremstyle{remark}
\newtheorem{rem}[thm]{Remark}

\newtheorem{ex}[thm]{Example}

\numberwithin{equation}{section}

\title[Regularity of complexes]{Regularity bounds for complexes and their homology}
\author{Hop D. Nguyen}
\address{Dipartimento di Matematica, Universit\`a di Genova, Via Dodecaneso 35, 16146 Genoa, Italy}
\address{Fachbereich Mathematik/Informatik, Institut f\"ur Mathematik, Universit\"at Osnabr\"uck, Albrectstr. 28a, 49069 Osnabr\"uck, Germany.}
\thanks{The author is grateful to the financial support of the CARIGE foundation.}
\email{ngdhop@gmail.com}

\begin{document}

\begin{abstract}
Let $R$ be a standard graded algebra over a field $k$. We prove an Auslander-Buchsbaum formula for the absolute Castelnuovo-Mumford regularity, extending important cases of previous works of Chardin and R\"omer. For a bounded complex of finitely generated graded $R$-modules $L$, we prove the equality $\reg L=\max_{i\in \Z} \{\reg H_i(L)-i\}$ given the condition $\depth H_i(L)\ge \dim H_{i+1}(L)-1$ for all $i<\sup L$. As applications, we recover previous bounds on regularity of $\Tor$ due to Caviglia, Eisenbud-Huneke-Ulrich, among others. We also obtain strengthened results on regularity bounds for $\Ext$ and for the quotient by a linear form of a module. 
\end{abstract}

\maketitle
\section{Introduction}
Let $(R,\mm)$ be a standard graded algebra over a field $k$ with the graded maximal ideal $\mm$, and $M$ a finitely generated graded $R$-module. The {\em relative Castelnuovo-Mumford of $M$} over $R$ is 
\[
\reg_R M=\sup\{j-i:\Tor^R_i(k,M)_j\neq 0\}.
\] 
This is a measure for the complexity of the $R$-module $M$. Looking at the free resolution of $M$ as an $R$-module, the regularity can be used to bound the degrees in which the free modules are generated. Castelnuovo-Mumford regularity can also be used to bound many other invariants; see, for example, \cite{BLS}, \cite{Cha}, \cite{CHH}. Closely related is the notion of absolute Castelnuovo-Mumford regularity of $M$, 
\[
\reg M=\sup\{i+j:H^i_{\mm}(M)_j\neq 0\},
\]
where $H^i_{\mm}(M)$ denotes the $i$-th local cohomology of $M$ with support in $\mm$ (by convention $\reg 0=-\infty$). The two notions are the same if $R$ is a polynomial ring over $k$, but in general $\reg_R M$ may be infinite for some non-zero $M$ unless $R$ is a Koszul algebra; see \cite{AE} and \cite{AP}. In contrast, $\reg M$ is always a finite number if $M\neq 0$, and we will mainly focus on the absolute regularity here. There is a large body of research on the absolute regularity; see, e.g., \cite{BM}, \cite{BLS}, \cite{Cha}, \cite{CHT}, \cite{EG}, \cite{Ko} and the references therein for more information. 

In \cite{Jor2}, \cite{Jor3}, P. J{\o}rgensen studies the absolute regularity of {\em complexes} over non-commutative graded algebras, using local cohomology. In this paper, we will complement his approach by proving that regularity of {\em bounded complexes} can be computed via $\Ext$ or Koszul homology. In particular, we show that over any algebra $R$ and any complex $M$ with bounded homology,
\begin{equation}
\label{Ext_formula}
\reg M = \sup\{i+j:\Ext^i_R(k,M)_j\neq 0\},
\end{equation}
and
\begin{equation}
\label{Koszul_formula}
\reg M = \sup\{j-i: H_i(\bsx,M)_j\neq 0\},
\end{equation}
where $\bsx=x_1,\ldots,x_n$ are linear forms such that all the Koszul homology $H_i(\bsx,M)$ have finite length. In the special case when $M$ is a module, the first statement follows from work of Chardin and Divaani-Aazar \cite[Theorem 2.15(5)]{ChD} which used other tools. The second statement was partly suggested by work of Avramov, Iyengar and Miller \cite[Theorem 7.2.3]{AIM}. One can think of the above results for regularity as analogues of the three equivalent formulas for the depth of $M$ (see also \cite{FI}):
\begin{enumerate}
\item $\depth M=\min\{i:H^i_{\mm}(M)\neq 0\}$;
\item $\depth M=\min\{i:\Ext^i_R(k,M)\neq 0\}$;
\item $\depth M=n-\max\{i: H_i(\bsx,M)\neq 0\}$, where $\bsx=x_1,\ldots,x_n$ are elements that generate $\mm$.
\end{enumerate}
This similarity is in fact a kind of starting point for this work. We prove the above formulas for regularity in Sections \ref{sect_exreg} and \ref{sect_kreg} (Theorems \ref{localreg_exreg} and \ref{thm_finite}).

The major concern in this paper is to unify various regularity bounds for Tor and Ext in the literature; see, e.g., \cite{BLS}, \cite{Ca}, \cite{EHU}, \cite{Cha}, \cite{ChD}. Simple as they are, formulas \eqref{Ext_formula} and \eqref{Koszul_formula} prove to be useful for this purpose. A helpful observation is that many regularity bounds for Tor and Ext can be rephrased as statements comparing regularity of complexes and that of their homology. Hence the typical contribution of this paper takes the following form. Below, the condition ``$G\in \dgrcatb{R}$" means that $G$ is a complex of graded $R$-modules such that $H_i(G)$ is finitely generated and is zero for almost all $i$. The notation $\inf G$ stands for $\inf \{i:H_i(G)\neq 0\}$.
\begin{thm}[Theorem \ref{thm_dim1}]
\label{thm_main1}
Let $G\in \dgrcatb R$ be a complex of graded $R$-modules such that $\dim H_i(G)\le 1$ for all $i>\inf G$. Then
\[
\reg G=\sup_{i\in \Z}\{\reg H_i(G)-i\}.
\]
\end{thm}
The last result is enough to recover various regularity bounds for Tor in \cite{Ca}, \cite{EHU}, \cite{Cha}, \cite{BLS}. For example, we obtain the following important case of a result due to Chardin.
\begin{thm}[See Chardin, {\cite[Theorem 0.2]{Cha}}]
\label{thm_tensor}
Let $R$ be a standard graded algebra over $k$. Let $M, N$ be finitely generated graded $R$-modules such that $\projdim_R M<\infty$. If moreover $\dim \Tor^R_i(M,N)\le 1$ for all $i\ge 1$ then
\[
\max_{i\ge 0}\{\reg \Tor^R_i(M, N)-i\} = \reg M+\reg N-\reg R.
\]
\end{thm}
Theorem \ref{thm_tensor} extends previous results on the regularity of $\Tor$ due to Caviglia \cite[Corollary 3.4]{Ca}, Eisenbud-Huneke-Ulrich \cite[Corollary 3.1]{EHU}. Our result is indeed stronger than \ref{thm_tensor} and is not covered by Chardin's method, see Remark \ref{rem_compare}. To derive Theorem \ref{thm_tensor} from Theorem \ref{thm_main1}, we set $G=M\dtensor{R}N$, the derived tensor product of $M$ and $N$. That the regularity of $M\dtensor{R}N$ equals $\reg M+\reg N-\reg R$ is clarified by an Auslander-Buchsbaum formula for regularity (Theorem \ref{AB_exreg}). The last formula illustrates the strength of formula \eqref{Ext_formula} and it extends \cite[Theorem 4.2]{Roe} and a major case of \cite[Theorem 5.3]{Cha}. 

As another application, we also obtain regularity bounds for $\Ext$, broadly extending works of Caviglia \cite{Ca}, Chardin and Divaani-Aazar \cite{ChD}; see Corollary \ref{cor_hom} and Remark \ref{rem_compare}. A new and readily comprehensible consequence of our method is the following
\begin{cor}
\label{cor_Homregularity}
Let $R$ be a standard graded polynomial ring over $k$, $N$ a finitely generated graded $R$-module and $I$ a homogeneous ideal of $R$. If $\dim \Hom_R(R/I,N)\le 1$ then
\[
\reg \Hom_R(R/I,N)\le \reg N.
\]
\end{cor}

While our method appears to be new and yields improvements to various known results on the regularity of $\Tor$ and $\Ext$, we would like to point out that the methods of Chardin \cite{Cha} and Eisenbud-Huneke-Ulrich \cite{EHU} have their own strength and produce significant results in other directions. Some further results concerning the regularity of various functors may be found in \cite{BLS}, \cite{Cha}, \cite{ChD}, \cite{CHH}.

The paper is organized as follow. We start by reviewing the necessary materials in Section \ref{sect_back}. In Section \ref{sect_exreg}, we study the version of regularity defined by $\Ext$, proving ultimately that it is equal to the regularity. We work out the Koszul homology approach to regularity \eqref{Koszul_formula} in Section \ref{sect_kreg}. A generalized Auslander-Buchsbaum formula for regularity is proved in Section \ref{Auslander_Buchsbaum}. In Section \ref{sect_bounding_Tor}, we prove the main result of our paper (Theorem \ref{thm_cmd1}) which generalizes Theorem \ref{thm_main1}. The proof of the main theorem essentially depends on the Koszul homology approach to regularity. From that we obtain in Section \ref{sect_appl} bounds for the regularity of $\Tor$ and $\Ext$ over an arbitrary base. Theorem \ref{thm_tensor} will be derived from \ref{cor_tensor} -- a result obtained by Chardin via a different approach. We also derive other consequences, for example that concerning regularity of a module modulo a homogeneous form; see Corollary \ref{cor_filter_regular}. We give an example which shows that the condition in the last result is necessary, and as a by product, we obtain a counterexample to a problem proposed by Caviglia in the list of Peeva and Stillman \cite[Problem 3.7]{PS}.
\section{Background}
\label{sect_back}
Let $k$ be a field. We will always use $(R,\mm,k)$ to denote a standard graded $k$-algebra $R$ with graded maximal ideal $\mm$. In particular, $R$ is commutative, $\N$-graded and generated by finitely many elements of degree $1$ over $R_0=k$. {\it Unless otherwise stated, modules over $R$ are graded, and morphisms between modules are degree-preserving}. Denote by $\grcat R$ be abelian category of complexes of graded modules over $R$. Modules are identified with complexes concentrated in degree $0$. The morphisms in $\grcat R$ are $R$-linear morphisms of complexes preserving degrees.  Its derived category (homotopy category) is denoted by $\dcat{\Grsf\,R}$ (respectively $\hgrcat R$). Isomorphisms in the derived category are signified by $\simeq$. The full subcategory of $\dcat{\Grsf\,R}$ consisting of complexes with finitely generated homology is denoted by $\dgrcat R$. A complex $M\in \dgrcat R$ is bounded above (respectively, below) if $H_i(M)=0$ for $i\gg 0$ ( resp., $i\ll 0$). The full subcategory of bounded above (below) complexes of $\dgrcat R$ is denoted by $\dgrcatn R$ (resp., $\dgrcatp R$). Denote by $\dgrcatb R$ the full subcategory of bounded complexes, namely those with $H_i(M)=0$ for $|i|\gg 0$. Similarly, we define the subcategories $\hgrcatn R, \hgrcatp R, \hgrcatb R$ of $\hgrcat R$. The derived functors of $\otimes$ and $\Hom$ are written as $-\dtensor{R}-$ and $\RHom{R}(-,-)$. We refer to \cite{CF} for a detailed treatment of homological algebra of complexes. For unexplained notations and terminology on commutative algebra, we refer to \cite{BH}.

The notation $\shift$ denotes the shift functor whose action on a complex $M$ is given by $(\shift M)_i=M_{i-1}$ for all $i\in \Z$.

For any complex $M\in \dgrcat R$, let $\inf M$ be the number $\inf\{i:H_i(M)\neq 0\}$. Define analogously $\sup M$. If $M\in \dgrcatp R$, define $\beta_{i,j}(M)=\dim_k \Tor^R_i(k,M)_j$; this is called the ($i,j$)-Betti number of $M$. Notice that any $M\in \hgrcatp R$ has a minimal graded free resolution $F\in \hgrcatp R$, namely a complex of graded free modules $F$ satisfying $\img(F_i\to F_{i-1})\subseteq \mm F_{i-1}$ for all $i\in \Z$ such that we have a quasi-isomorphism $F\to M$.

The bigraded Poincar\'e series of $M$ is the formal power series
\[
P^R_M(t,y)=\sum_{i\in \Z}\sum_{j\in \Z}\beta_{i,j}(M)t^iy^j \in \Q[[t,y]].
\]
Any complex $M \in \hgrcatn R$ has a minimal injective resolution $I\in \hgrcatn R$. Concretely, we have a quasi-isomorphism $M\simeq I$ and $I$ is a complex of injective objects in $\grcat R$ such that $\Ker(I_{j}\to I_{j-1})\subseteq I_j$ is a graded-essential extension for every $j$. We will call injective modules in $\grcat R$ {\it ${}^*$injective modules}.

The {\em bigraded Bass series} of $M\in \dgrcatn R$ is defined by
\[
I^M_R(t,y)=\sum_{i\in \Z}\sum_{j\in \Z}\mu^{i,j}_R(M)t^iy^{-j} \in \Q[[t,y]],
\]
where $\mu^{i,j}_R(M)=\dim_k \Ext^i_R(k,M)_j$. Note that the power of $y$ is $-j$, not $j$ like in the Poincar\'e series. The reason for this is to extend the usual Bass series formula from the local situation to the graded one; see Proposition \ref{Bass_series}.

By the graded analogue of \cite[Proposition 1.5.3]{AF}, we have the following result. The proof in {\em loc. cit.} transfers directly to the graded situation, but we include the argument for completeness.
\begin{prop}
\label{Bass_series}
Let $(R,\mm,k)\to (S,\nn,l)$ be a homomorphism of standard graded algebras. Let $M\in \dgrcatp R$, $N\in \dgrcatn S$ be complexes. Then
\[
I^{\RHom R(M,N)}_S(t,y)=P^R_M(t,y)I^N_S(t,y).
\]
\end{prop}
\begin{proof}
We have the following sequence of graded isomorphisms
\begin{align*}
\RHom S(l,\RHom R(M,N)) &\simeq \RHom S(l\dtensor R M, N)\\
                           &\simeq \RHom S((M\dtensor R k)\otimes_k l, N)\\
                           &\simeq \Hom_k(M\dtensor R k, \RHom S(l,N))\\
                           &\simeq \Hom_k(M\dtensor R k, k) \otimes_k \RHom S(l,N).
\end{align*}
Taking the homology and applying K\"unneth's formula, we get
\[
\Ext^*_S(l,\RHom R(M,N)) \cong \Hom_k(\Tor^R_*(M,k),k)\otimes_k \Ext^*_S(l,N).
\]
In particular,
\begin{equation}
\label{eq_Bass}
\mu^{i,j}_S(\RHom R(M,N))= \sum_{u\in \Z}\sum_{v\in \Z}\beta^R_{u,v}(M)\mu^{i-u,j+v}_S(N).
\end{equation}
This gives the desired formula.
\end{proof}
\subsection{Local cohomology of complexes}
\label{sect_derived_lc}
In this section, we recall the construction of the local cohomology functor in the derived category. We follow the exposition in Foxby and Iyengar \cite{FI} and Lipman \cite{Lip}.

Let $M$ be a (not necessarily finitely generated) graded {\em $R$-module}. The torsion functor $\Gamma_{\mm}$ gives the following submodule of $M$:
\[
\Gamma_{\mm}(M)=\{a\in M: \mm^ia=0 ~\textnormal{for some $i>0$}\}.
\]
This extends to a functor $\Gamma_{\mm}: \dcatgr R \to \dcatgr R$, which is again called the torsion functor. This is a left-exact functor and its right derived functor is written as $\dlc_{\mm}(-)$. For a {\em complex} $M\in \dcatgr R$ and each $i\in \Z$, $H^i_{\mm}(M)=H_{-i}(\dlc_{\mm}(M))$ is an $\mm$-torsion module, called the $i$-th local cohomology of $M$ with support in $\mm$. The following result is immediate from \cite[Corollary 3.4.3]{Lip}.
\begin{lem}
\label{lc_base_change}
Let $(R,\mm,k)\to (S,\nn,k)$ be a surjection of standard graded algebras. Let $M\in \dcatgr S$ be a complex. Then for each $i\in \Z$, we have an isomorphism of graded $R$-modules
\[
H^i_{\mm}(M)\cong H^i_{\nn}(M).
\]
\end{lem}
\subsection{Absolute Castelnuovo-Mumford regularity}
\label{sect_CMreg}
Let $M\in \dgrcat R$ be a complex. We define the {\em absolute Castelnuovo-Mumford regularity} of $M$ by the formula
\[
\reg M=\sup\{i+j: H^i_{\mm}(M)_j\neq 0\}.
\]
By convention, $\reg 0=-\infty$. If $M\in \dgrcatp R$ is a bounded below complex, the {\em relative Castelnuovo-Mumford regularity of $M$ over $R$} is defined by
\[
\reg_R M=\sup\{j-i: \Tor^R_i(k,M)_j\neq 0\}.
\]
We say that $R$ is a Koszul algebra if $\reg_R k=0$. Assume now that $R$ is a standard graded polynomial ring over $k$. Since $R$ is a Koszul algebra and $\reg R=0$, by a result of P. J{\o}rgensen \cite[Corollary 2.8]{Jor2}, we have $\reg_R M=\reg M$ for any $M\in \dgrcatb R$.
\begin{rem}
Note that in \cite{Jor3}, P. J{\o}rgensen used different names and notations for $\reg_R M$ and $\reg M$. For example, the first is called the $\Ext$-regularity of $M$ over $R$ and denoted by $\textnormal{Ext.reg}\, M$. This is not to be confused with the notion $\exreg_R M$ defined below in Section \ref{sect_exreg}. The main result of Section \ref{sect_exreg} will show that the latter invariant is equal to the {\em absolute} regularity if $M$ is bounded. 
\end{rem}
The comparison between absolute and relative regularity is given by the following result.
\begin{thm}[P. J{\o}rgensen, {\cite[Theorems 2.5, 2.6]{Jor3}}]
\label{comparison_reg}
If $M\in \dgrcatb R$ is a non-trivial complex then there are inequalities
\begin{align*}
\reg_R M &\le \reg M + \reg_R k,\\
\reg M &\le \reg_R M+\reg R.
\end{align*}
In particular, if $R$ is a Koszul algebra then $\reg_R M\le \reg M$.
\end{thm}
See the proof of \cite[Theorem 3.1]{Roe} for a simpler exposition in the case $M$ is a module.

\section{Regularity via Ext modules}
\label{sect_exreg}
For convenience of exposition, let us introduce the following notion.
\begin{defn}
Let $M\in \dcatgr R$ be a bounded above complex of graded $R$-modules. We define the invariant $\exreg_R M$ of $M$ as follow
\[
\exreg_R M=\sup\{i+j:\Ext^i_R(k,M)_j\neq 0\}.
\]
By convention, if $M\simeq 0$, we let $\exreg_R M=-\infty$.
\end{defn}
\begin{rem}
(i) The number $\exreg_R M$ is well-defined in the derived category since it can be computed from the minimal ${}^*$injective resolution of $M$; see Proposition \ref{exreg_inj_res}.

(ii) Note that we always have $\exreg_R k = 0$; this comes from the following observation: if $F$ is the minimal graded free resolution of $k$ then $\Ext^i_R(k,k)\cong \Hom_R(F_i,k)$.
\end{rem}
By standard arguments, we infer
\begin{lem}
\label{exreg_seq}
If $0\to M_1\to M_2 \to M_3\to 0$ be a short exact sequence of (not necessarily finitely generated) graded $R$-modules, then we have
\begin{align*}
\exreg_R M_2 &\le \max\{\exreg_R M_1,\exreg_R M_3\},\\
\exreg_R M_1 &\le \max\{\exreg_R M_2,\exreg_R M_3+1\},\\
\exreg_R M_3 &\le \max\{\exreg_R M_2,\exreg_R M_1-1\}.
\end{align*}
\end{lem}
For any ${}^*$injective $R$-module $I$, clearly $\exreg_R I=\sup\{j:\Hom_R(k,I)_j\neq 0\}$.
\begin{prop}
\label{exreg_inj_res}
Let $I_{\pnt}$ be the minimal ${}^*$injective resolution of a bounded above complex $M\in \dcat{\Grsf\,R}$. Then
\[
\exreg_R M =\sup_{i\in \Z}\{\exreg_R I_i-i\}.
\]
\end{prop}
\begin{proof}
This is immediate since $\Hom_R(k,I_i) \cong \Ext^{-i}_R(k,M)$ for all $i\in \Z$.
\end{proof}
\begin{ex}
\label{exreg_polynomial}
Let $R=k[x_1,\ldots,x_n]$ (where $n\ge 1$) be a standard graded polynomial ring. Then $\Ext^i_R(k,R)=0$ for $i\le n-1$ and
\[
\Ext^n_R(k,R)\cong k(n).
\]
Therefore $\exreg_R R=0$.
\end{ex}
The following result is an analogue of \cite[Proposition 6.1]{NgV}.
\begin{lem}
\label{exreg_homology}
For any complex $G\in \dgrcatn R$, there is an inequality
\[
\exreg G \le \sup_{i\in \Z}\{\exreg H_i(G)-i\}.
\]
\end{lem}
The inequality is an equality if the differential of $G$ is zero.
\begin{proof}
Replacing $G$ by $\shift^{-m}G$ then both sides increase by $m$ for every $m\in \Z$. So we can assume that $\sup H(G)=0$.

It is harmless to replace $G$ by its minimal ${}^*$injective resolution $I$, which can be chosen such that $\sup I=\sup H(G)=0$. Denote $B_j=\img(I_{j+1} \to I_j), Z_j=\Ker(I_j\to I_{j-1}), H_j=H_j(G)$. We will show that for all $j\le 0$,
\begin{equation}
\label{exreg_ineq}
\exreg_R I_j \le \max_{j\le i\le 0}\{\exreg_R H_{j+i}-i\}.
\end{equation}
Indeed, there are short exact sequences
\begin{align}
0\to B_i \to Z_i \to H_i \to 0, \label{sequence_B} \\
0\to Z_i \to I_i \to B_{i-1} \to 0 \label{sequence_Z}.
\end{align}
Note that for each $i$, \eqref{sequence_Z} is the beginning of the minimal ${}^*$injective resolution of $Z_i$, therefore from Proposition \ref{exreg_inj_res},
\begin{equation}
\label{exreg_Z}
\exreg Z_i=\max\{\exreg_R I_i, \exreg_R B_{i-1}+1\}.
\end{equation}

Applying Lemma \ref{exreg_seq} for \eqref{sequence_B} and \eqref{sequence_Z}, we obtain
\begin{align*}
\exreg_R Z_j &\le \max\{\exreg_R H_j, \exreg_R B_j\}\\
              &\le \max\{\exreg_R H_j, \exreg_R Z_{j+1}-1\},
\end{align*}
where the second inequality follows from \eqref{exreg_Z}. Starting from $Z_0=H_0$, an easy induction gives for any $j<0$
\begin{equation*}
\exreg_R Z_j \le \max_{0\le i\le -j}\{\exreg_R H_{j+i}-i\}.
\end{equation*}
So \eqref{exreg_ineq} follows since from \eqref{exreg_Z}, $\exreg_R I_j \le \exreg_R Z_j$.

Now from \eqref{exreg_ineq}, for each $j\le 0$,
\begin{align*}
\exreg_R I_j -j \le \max_{j\le i\le 0}\{\exreg_R H_i-i\}.
\end{align*}
Hence taking suprema over all $j\le 0$, we finish the proof of the inequality.

Assume that $G$ has differential $0$, hence $H_i(G)=G_i$. For each $i$, let $L^i_*$ be the minimal ${}^*$injective resolution of $G_i$. Then $\bigoplus \shift^i L^i_*$ is the minimal ${}^*$injective resolution of $G$. Hence
\begin{align*}
\exreg_R G &=\sup_{j\le 0} \{\exreg_R (\bigoplus_{j\le i\le 0} L^i_{j-i})-j\}\\
            &=\sup_{i\le 0} \{ \sup_{j\le i}\{\exreg_R L^i_{j-i}-(j-i)-i\} \}\\
            &=\sup_{i\le 0} \{\exreg_R G_i-i\},
\end{align*}
as desired.
\end{proof}
That $\exreg_R M$ is always a finite invariant in cases of interest is confirmed by the next result. 
\begin{prop}[Ring independence]
\label{surjection}
Let $R'\to R$ be a surjection of standard graded $k$-algebras. Then for each non-trivial $M\in \dgrcatn R$, we have
\[
\exreg_R M = \exreg_{R'} M.
\]
In particular, if $M\in \dgrcatb R$ and $M\not\simeq 0$ then $\exreg_R M$ is a finite number.
\end{prop}
Firstly, we have
\begin{lem}
\label{cor_reg_Hom}
Let $R'\to R$ be a surjection of standard graded $k$-algebras, $N\in \dgrcatn R$ a complex and $M$ a finitely generated $R'$-module. Then we have 
$$
\exreg_R \RHom {R'}(M,N)=\exreg_R N-\indeg M.
$$ 
Here $\indeg M=\inf\{i:M_i\neq 0\}$.
\end{lem}
\begin{proof}
By the proof of  \ref{Bass_series}, for all $i,j\in \Z$, we have an equation 
\[
\mu^{i,j}_R(\RHom {R'}(M,N))= \sum_{u\in \Z}\sum_{v\in \Z}\beta^{R'}_{u,v}(M)\mu^{i-u,j+v}_R(N).
\]
Taking suprema from the equality $i+j=(i-u+j+v) +(u-v)$, we obtain $\exreg_R \RHom {R'}(M,N)=\exreg_R N-\indeg M$, as desired.
\end{proof}
\begin{proof}[Proof of Theorem \ref{surjection}]
Let $I$ be the minimal ${}^*$injective resolution of $M$ over $R'$. For each $i\in \Z$, $H_{-i}(\RHom{R'}(k,M)) \cong \Ext^i_{R'}(k,M)$ as $R'$-modules. Since $R'\to R$ is surjective, and $\Ext^i_{R'}(k,M)$ is a $k$-vector space, we also have
\[
H_{-i}(\RHom{R'}(k,M)) \cong \Ext^i_{R'}(k,M)
\]
as $R$-modules, for all $i\in \Z$.

{\em Step 1}. We will show that $\exreg_R M\le \exreg_{R'} M$. Applying Lemma \ref{exreg_homology} for the complex of $R$-modules $\RHom{R'}(k,M)$, we see that
\[
\exreg_R \RHom{R'}(k,M) \le \sup_{i\in \Z}\{\exreg_R \Ext^i_{R'}(k,M)+i\}.
\]
If $r_{R'}^i=\sup\{j:\Ext^i_{R'}(k,M)_j \neq 0\}$ for each $i\in \Z$, then
\[
\exreg_R \Ext^i_{R'}(k,M)=\exreg_R k(-r_{R'}^i)=r_{R'}^i.
\]
Finally, using Lemma \ref{cor_reg_Hom}, one has
\begin{align*}
\exreg_R M &= \exreg_R \RHom{R'}(k,M)\\
            &\le \inf_{i\in \Z}\{r_{R'}^i+i\}\\
            &=\exreg_{R'} M.
\end{align*}

{\em Step 2}. It remains to show that $\exreg_{R'} M \le \exreg_R M$. There is a standard change of rings spectral sequence for $\Ext$:
\[
E^{p,q}_2=\Ext^p_R(\Tor^{R'}_q(k,R),M) \Rightarrow \Ext^{p+q}_{R'}(k,M).
\]
Therefore by usual arguments
\[
I^M_{R'}(t,y) \preccurlyeq I^M_R(t,y)P^{R'}_R(t,y),
\]
and hence
\[
\exreg_{R'} M \le \exreg_R M.
\]
The first part of the theorem is proved. 

For the second part, choose $R'$ to be a polynomial ring over $k$ which surjects onto $R$. If $M\in \dgrcatb R$ then $\exreg_{R'} M$ is finite since $\Ext^i_{R'}(k,M)=0$ unless $-\sup M\le i\le \dim R'-\inf M$.
\end{proof}
We are in the position to show that $\exreg$ is the same as the absolute regularity.
\begin{thm}
\label{localreg_exreg}
Let $R$ be a standard graded $k$-algebra and $M\in \dgrcatb R$. Then $\reg M=\exreg_R M$.
\end{thm}
\begin{proof}
Let $R'$ be a standard graded polynomial ring over $k$ which surjects onto $R$. Then we know that $\reg M=\reg_{R'}(M)$ by Lemma \ref{lc_base_change} and $\exreg_R M=\exreg_{R'} M$ by Proposition \ref{surjection}. Therefore it is enough to assume that $R$ is a polynomial ring $k[x_1,\ldots,x_n]$. The conclusion follows from the isomorphisms $\Tor^R_i(k,M)\cong \Ext^{n-i}_R(k,M)(-n)$ for all $i\in \Z$. 

The last isomorphism can be proved as follow: let $K$ be the Koszul complex on a minimal set of generators of the maximal ideal of $R$. Since $K$ is a free complex, we have isomorphisms
\begin{align*}
k\dtensor{R} M \simeq K\otimes_R M & \simeq \shift^n\Hom_R(K,R(-n))\otimes_R M\\
                                      & \simeq \shift^n\Hom_R(K,R(-n)\otimes_R M)\\
                                      & \simeq \shift^n\Hom_R(K,M(-n)).
\end{align*}
The second isomorphism is due to the self-duality of $K$. This concludes the proof.
\end{proof}
\section{Regularity via Koszul homology}
\label{sect_kreg}
Let $(R,\mm,k)$ be a standard graded algebra and let $M\in \dgrcatp R$ be a bounded below complex. Let $\bsx=x_1,\ldots,x_n$ be a sequence of degree $1$ homogeneous elements of $R$. Denote by $K[\bsx;R]$ the Koszul complex on the sequence $\bsx$. Let $K[\bsx;M]=K[\bsx;R]\dtensor{R} M$ be the Koszul complex on the sequence $\bsx$ with coefficients in $M$. Then $H_i(\bsx,M)=H_i(K[\bsx;M])$ is a graded $R$-module for each $i\in \Z$. We denote
\[
\kreg^{\bsx}_R M = \sup\{j-i:H_i({\bsx},M)_j \neq 0\}.
\]
Assume that furthermore, $\bsx$ minimally generates the maximal ideal $\mm$. Then as $K[\bsx;M]$ is independent of the choice of $\bsx$ up to isomorphisms of complexes, we can introduce the following notion (where ``k" will stand for ``Koszul").
\begin{defn}
We denote by $\kreg_R M$ the invariant of $M$ defined by $\kreg^{\bsx}_R M$, where $\bsx$ is a minimal system of generators  of $\mm$ consisting of elements with degree $1$.
\end{defn}
We recall the following easy and well-known result.
\begin{lem}
\label{lem_kreg}
Let $(R,\mm,k)$ be a standard graded algebra and $M\in \dgrcatb R$. Then there is an equality
\[
\kreg_R M=\reg M.
\]
\end{lem}
\begin{proof}
Let $\bsx=x_1,\ldots,x_n$ be a minimal system of generators of $\mm$, where $\deg x_i=1$. Let $R'=k[y_1,\ldots,y_n]$ be a polynomial ring and $R'\to R$ be the canonical projection. Denote $\bsy=y_1,\ldots,y_n$, in $\dgrcat{R'}$, we have the following isomorphisms 
\[
K[\bsy;M]=K[\bsy;R']\otimes_{R'} M \simeq K[\bsy;R']\otimes_{R'} R \otimes_R M \simeq K[\bsx;M].
\]
Therefore $\kreg_R M=\kreg_{R'} M=\reg M$. The last equality holds because of the isomorphism $\Tor^{R'}_i(k,M)\cong H_i(\bsy,M)$ for all $i\in \Z$.
\end{proof}
The main result of this section is the next theorem. It establishes the Koszul homology approach to the absolute regularity, namely equality \eqref{Koszul_formula} in the introduction. The proof is similar to that of \cite[Theorem 7.2.3]{AIM}. Note that one could also have deduced the theorem from Corollary \ref{cor_Koszul}, but we decide to keep the next proof to provide another perspective to this theorem.
\begin{thm}
\label{thm_finite}
Let $\bsy=y_1,\ldots,y_n$ be elements of degree $1$, $M\in \dgrcatb R$ such that $H_i(\bsy,M)$ has finite length for each $i$. For any linear form $y_{n+1}$, denote $\bsy'=y_1,y_2,\ldots,y_{n+1}$, then we have
\[
\kreg^{\bsy}_R M=\kreg^{\bsy'}_R M.
\]
In particular, $\reg M=\kreg^{\bsy}_R M$.
\end{thm}
\begin{proof}
Denote by $\ell(N)$ the number $\dim_k N$ for an $R$-module of finite length $N$. Denote $z=y_{n+1}$, $K=K[\bsy;M]$ and $K'=K[\bsy';M]$. We have an exact triangle in $\dgrcat R$
\[
K(-1) \xrightarrow{\cdot z} K \to K' \to.
\]
By the hypothesis, for each $i$, $H_i(K)$ has finite length. Moreover, as $M$ is bounded, $H_i(K)=0$ for all but finitely many $i$. So there exists a number $t\ge 1$ such that $z^tH_i(K)=0$ for all $i$. Using \cite[Lemma 4.9]{NgV}. we get
\begin{align*}
&\ell(H_i(K')_j) \le \ell(H_i(K)_j)+\ell(H_{i-1}(K)_{j-1}),\\
&\ell(H_i(K)_j) \le \sum_{r=0}^{t-1}\ell(H_{i+1}(K')_{j+r+1}).
\end{align*}
From the first inequality, we conclude that $\kreg^{\bsy'}_R(M)\le \kreg^{\bsy}_R(M)$, while from the second one, we obtain $\kreg^{\bsy}_R(M)\le \kreg^{\bsy'}_R(M)$. So $\kreg^{\bsy'}_R(M)= \kreg^{\bsy}_R(M)$, as desired.

For the second part, let $\bsx$ be a minimal system of generators of $\mm$, then using the first part
\[
\kreg_R M=\kreg_R^{\bsx} M=\kreg_R ^{\bsx \sqcup \bsy}M=\kreg_R^{\bsy} M.
\]
Together with Lemma \ref{lem_kreg}, the proof is therefore completed.
\end{proof}
\section{Auslander-Buchsbaum formula for regularity}
\label{Auslander_Buchsbaum}
For a complex $N\in \dgrcatp R$, the projective dimension of $N$ is defined as 
\[
\projdim_R N=\inf\{r:\textnormal{there exists a free resolution $F$ of $N$ such that $F_i=0$ for $i>r$}\}.
\] 
It is always the case that $\projdim_R N=\sup(k\dtensor{R}N)$; see, e.g., \cite[Theorem 5.2.3]{CF}. We have the following formula of Auslander-Buchsbaum type for absolute Castelnuovo-Mumford regularity. This is an analogue of the depth formula \cite[Theorem 5.2.12]{CF}, \cite[Theorem 4.1]{Iy} and the proof uses the same idea.
\begin{thm}[Auslander-Buchsbaum formula]
\label{AB_exreg} 
Let $R$ be a standard graded $k$-algebra, $M, N$ complexes in $\dgrcatb R$. If $N$ has finite projective dimension then
\begin{align*}
\reg (M\dtensor RN) &=\reg M + \reg_R N\\
                  &=\reg M + \reg N -\reg R.
\end{align*}
In particular, $\reg N=\reg R +\reg_R N.$
\end{thm}

\begin{proof}
By Theorem \ref{localreg_exreg}, we only need to establish the corresponding equalities with $\exreg$ in place of regularity.
Since $\projdim_R N<\infty$, by tensor evaluation \cite[Theorem 4.4.5]{CF},
\begin{align*}
\RHom R(k,M\dtensor RN) &\simeq \RHom R(k,M)\dtensor R N\\
                               &\simeq (\RHom R(k,M)\dtensor k k)\dtensor R N\\
                               &\simeq \RHom R(k,M)\dtensor k(k\dtensor R N)\\
                               &\simeq \RHom R(k,M)\otimes_k (k\dtensor R N).
\end{align*}
Taking homology and using the K\"unneth's formula, we have for all $i,j\in \Z$,
\[
\mu^{i,j}(M\dtensor RN)=\sum_s\sum_u \mu^{s,u}(M)H_{-i+s}(k\dtensor R N)_{j-u}.
\]
From the equality $i+j=(s+u) + (j-u -(-i+s))$, we get by computing suprema
\[
\exreg_R M\dtensor RN = \exreg_R M + \sup\{j-i: H_i(k\dtensor RN)_j\neq 0\}=\exreg_R M+\reg_R N.
\]
The second equation holds because of the definition of $\reg_R N$. This gives the first equality. As a corollary, choosing $M=R$, we obtain
\[
\exreg_R N = \exreg_R R +\reg_R N.
\]
Therefore the second equality follows.
\end{proof}
Consequently, we obtain the following result, which covers \cite[Theorem 4.2]{Roe} and \cite[Theorem 5.3]{Cha} for algebras over a field.
\begin{cor}
\label{AB_reg}
Let $R$ be a standard graded $k$-algebra. If $M \in \dgrcatb R$ satisfies $\projdim_R M<\infty$, then
\[
-\reg_R M + \reg M =\reg R.
\]
\end{cor}
\begin{ex}
\label{AB_Gorenstein}
Let $R$ be a Gorenstein $k$-algebra. Then for every finitely generated graded $R$-module $M$ with $\projdim_R M<\infty$, we have
\[
-\reg_R M + \reg M = \reg R=\dim R+a(R),
\]
where $a(R)$ is the $a$-invariant of $R$. 

This holds since $\Ext^i_R(k,R)=\begin{cases}
0, &\textnormal{if $i<\dim R$},\\
k(-a(R)),&\textnormal{if $i=\dim R$}.
\end{cases}$
\end{ex}
Note that even if $R$ is Koszul, it may still happen that $\reg M - \reg_R M \neq \reg R$ for some $R$-module $M$. For example, take $R=k[x,y]/(x^2,y^2), M=k$, then $\reg R=2$ while $-\reg_R k+\reg k=0$. So the formula in Corollary \ref{AB_reg} is not true for all modules of finite relative regularity.

\begin{cor}
Let $R\to S$ be a surjection of standard graded $k$-algebras such that $\projdim_R S<\infty$. Let $M\in \dgrcatb{S}$ be such that $\projdim_S M<\infty$. Then $\reg_R S+\reg_S M=\reg_R M$.
\end{cor}
\begin{proof}
Since $\projdim_R S<\infty$, the complex $k\dtensor{R}S$ has bounded homology. Denote $N=k\dtensor{R}S$. Using Theorem \ref{AB_exreg} we have 
\begin{align*}
\reg_R M &=\kreg_R (k\dtensor{R}M)=\kreg_S (N \dtensor{S}M)\\
         &=\kreg_S N+\reg_S M=\reg_R S+\reg_S M.
\end{align*}
Another proof: since $\projdim_R S<\infty$ and $\projdim_S M<\infty$, also $\projdim_R M<\infty$. Hence
\begin{align*}
\reg_R S+\reg_S M&=\exreg_R S-\exreg_R R+\exreg_S M-\exreg_S S\\
                 &=\exreg_S M-\exreg_R R = \exreg_R M-\exreg_R R\\
                 &=\reg_R M.
\end{align*}
\end{proof}

\section{Regularity bounds for homology of bounded complexes}
\label{sect_bounding_Tor}
Now we come to the main application of this paper, which concerns regularity bounds for complexes in terms of the regularity of their homology. The main result in this section is Theorem \ref{thm_cmd1}, which provides a sufficient condition for a complex $L\in \dgrcatb R$ to have $\reg L=\max_{i\in \Z}\{\reg H_i(L)-i\}$. We begin by a simple bound for the regularity of a complex, which we learned from Srikanth Iyengar. Observe that the result is an analog of the inequality \cite[Proposition 2.7(i)]{Iy} for depth. Through out, $(R,\mm,k)$ is a standard graded algebra. 
\begin{prop}
\label{prop_degreewise}
Let $G\in \dgrcatb R$ be a complex. Then we always have
\[
\reg G\le \sup_{i\in \Z}\{\reg G_i-i\}.
\]
\end{prop}
\begin{proof}
We can assume that $R$ is a polynomial ring and $\inf G=0$. Denote $t_i(G)=\sup\{j:\Tor^R_i(k,G)_j\neq 0\}$. From the discussion in Section \ref{sect_CMreg}, we obtain $\reg G = \sup_{i\ge 0} \{(t_i(G)-i\}$.

It is enough to show that $t_i(G) \le \max\{t_i(G_0),t_{i-1}(G_1),\ldots,t_0(G_i)\}$ for each $i\ge 0$. Indeed, if this is the case, then 
\begin{align*}
t_i(G)-i &\le \max_{0\le j\le i}\{t_{i-j}(G_j)-(i-j)-j\}\\
                    &\le \max_{0\le j\le i}\{\reg G_j -j\},
\end{align*}
hence $\reg G \le \sup_{i\in \Z}\{\reg G_i -i\}.$

Let us denote $X^i=G_{\ge i}$. Note that $t_i(X^j)=0$ for $j\ge i+1$, since the minimal graded free resolution $F$ of $X^j$ can be chosen such that $F_t=0$ for $t\le j-1$.

Now for each $j=0,\ldots,i,$ we have an exact sequence in the abelian category $\grcat R$
\[
0\to \shift^{j}G_j \to X^j \to X^{j+1} \to 0.
\]
Therefore 
\[
t_i(X^j)\le \max\{t_i(\shift^jG_j),t_i(X^{j+1})\}=\max\{t_{i-j}(G_j),t_i(X^{j+1})\}.
\]
Summing up,
\begin{align*}
t_i(G) & \le \max \{t_i(G_0), t_i(X^1)\} \\
         & \le \max \{t_i(G_0), t_{i-1}(G_1),t_i(X^2)\} \\
         & \le \cdots \\
         & \le \max \{t_i(G_0),t_{i-1}(G_1),\ldots,t_0(G_i)\},         
\end{align*}
as claimed.
\end{proof}
The main result of this section as well as this paper is
\begin{thm}
\label{thm_cmd1}
Let $L \in \dgrcatb R$ be such that $\depth H_i(L)\ge \dim H_{i+1}(L)-1$ for all $i<\sup L$. Then we have 
$\reg L= \sup_{i\in \Z}\{\reg H_i(L)-i\}.$
\end{thm}
The condition ``$\depth H_i(L)\ge \dim H_{i+1}(L)-1$ for all $i<\sup L$" of Theorem \ref{thm_cmd1} cannot be weakened, see Corollary \ref{cor_filter_regular} and Remark \ref{rem_linear}. The proof of Theorem \ref{thm_cmd1} is not straightforward. The idea is essentially using the Koszul homology approach \ref{thm_finite} and working with the ``partial Koszul regularity". The main ingredient in the proof of Theorem \ref{thm_cmd1} is the following special case, which we will treat first.
\begin{thm}
\label{thm_dim1}
Let $G\in \dgrcatb R$ be a complex such that $\dim H_i(G)\le 1$ for all $i>\inf G$. Then
$
\reg G=\sup_{i\in \Z} \{\reg H_i(G)-i\}.
$
\end{thm}
The key step in the proof of this result is accomplished by the following lemma. The latter is of the same flavour as \cite[Theorem 12.1]{Cha2}.
\begin{lem}
\label{lem_smalldim}
Let $G \in \dgrcatb R$ be a complex such that $\dim_R H_i(G)\le i-s$ for all $i>s=\inf G$. Then there is an inequality 
\[
\reg G \ge \reg H_s(G)-s.
\]
\end{lem}
\begin{proof}
It is harmless to assume that $R$ is a polynomial ring and $\inf G=0$. Replacing $G$ by a minimal free resolution of it, we can assume that $G$ is a minimal complex of free $R$-modules
\[
\ldots \to G_j \to G_{j-1} \to \ldots \to G_1 \to G_0\to 0,
\]
and $\reg G=\sup_{i\ge 0}\{\reg G_i-i\}$.  

We have to show that $\reg H_0(G) \le \reg G$. Denote $B_i=\img(G_{i+1}\to G_i), H_i=H_i(G)$. For any $R$-module $P$, denote $\reg_m(P)=\sup\{i+j:H^i_{\mm}(P)_j\neq 0 ~\textnormal{and $i\ge m$}\}$; so $\reg P=\reg_0(P)$. Since $\dim P \le n=\dim R<\infty$, we also have $\reg_m P=-\infty$ for $m>n$.

For each $i\ge 0$, there are exact sequences
\begin{align*}
0\to B_i \to G_i \to G_i/B_i\to 0,\\
0\to H_{i+1} \to G_{i+1}/B_{i+1} \to B_i \to 0,
\end{align*}
so letting $i=0$, we get
\begin{align*}
\reg H_0 \le \max\{\reg G_0, \reg_1 B_0-1\},\\
\reg_1 B_0 \le\max\{\reg_1 G_1/B_1,\reg_2 H_1-1\}.
\end{align*}
Now $\dim H_1\le 1$, so $\reg_2 H_1=-\infty$. Hence $\reg_1 B_0\le \reg_1 G_1/B_1$.

Denote by $G_{\ge i}$ the truncated complex
\[
\ldots \to G_j \to G_{j-1} \to \ldots \to G_{i+1} \to G_i \to 0.
\] 
We have $G_1/B_1=H_0(\shift^{-1}G_{\ge 1})$. Therefore
\[
\reg H_0(G) \le \max\{\reg G_0,\reg_1 H_0(\shift^{-1}G_{\ge 1})-1\}.
\]
Using the condition $\dim H_2\le 2$ and arguing similarly, we get
\[
\reg_1 H_0(\shift^{-1}G_{\ge 1}) \le \max\{\reg_1 G_1,\reg_2 H_0(\shift^{-2}G_{\ge 2})-1\},
\]
which yields
\[
\reg H_0(G) \le \max\{\reg G_0,\reg_1 G_1-1,\reg_2 H_0(\shift^{-2}G_{\ge 2})-2\}.
\]
Continuing in this manner, finally
\begin{align*}
\reg H_0 &\le \max\{\reg G_0,\reg G_1-1,\ldots,\reg_n G_n-n,\reg_{n+1}H_0(\shift^{-n-1}G_{\ge n+1})-n-1\}\\
            &\le \reg G.
\end{align*}
The second inequality holds since $\reg_{n+1}(P)=-\infty$ for all $R$-module $P$. This finishes the proof of the lemma.
\end{proof}
\begin{proof}[Proof of Theorem \ref{thm_dim1}]
Again we can assume that $R$ is a polynomial ring, $G$ is a minimal complex of free $R$-modules
\[
\ldots \to G_d \to G_{d-1} \to \ldots \to G_1 \to G_0\to 0,
\]
and $\reg G=\sup_{i\ge 0}\{\reg G_i-i\}$.

By \cite[Proposition 6.1]{NgV}, it holds that $\reg G\le \sup_{i\ge 0}\{\reg H_i(G)-i\}$. It remains to show that $\reg H_i(G)-i\le \reg G$ for all $i\ge 0$. 

Since $\dim H_i(G)\le 1\le i$ for all $i\ge 1$, Lemma \ref{lem_smalldim} shows that $\reg H_0(G)\le \reg G$.

Fix $1\le i\le d$. Since $\dim H_j(G)\le 1$ for all $j>i$, applying Lemma \ref{lem_smalldim} for the complex $G_{\ge i}$, we get 
\begin{equation}
\label{ineq_G/B}
\reg (G_i/B_i)=\reg H_0(\shift^{-i}G_{\ge i}) \le \sup_{j\ge i}\{\reg G_j-j+i\}.
\end{equation}
Denote $s_i(N)=\max\{j:H^i_{\mm}(N)_j\neq 0\}$ for any $R$-module $N$. From the exact sequences
\begin{align*}
&0\to H_i \to G_i/B_i \to B_{i-1}\to 0,\\
&0\to B_{i-1} \to G_{i-1},
\end{align*}
and the associated exact sequences of local cohomology, we have 
\begin{align*}
s_1(H_i)&\le \max\{s_1(G_i/B_i),s_0(B_{i-1})\},\\
s_0(H_i)&\le s_0(G_i/B_i)\le \reg(G_i/B_i),\\
s_0(B_{i-1})&\le s_0(G_{i-1}) .
\end{align*}
Therefore $s_1(H_i)\le \max\{s_1(G_i/B_i),s_0(G_{i-1})\} \le \max\{\reg (G_i/B_i)-1, \reg G_{i-1}\}$.

But $\dim H_i\le 1$, therefore 
\[
\reg H_i=\max\{s_1(H_i)+1,s_0(H_i)\} \le \max\{\reg (G_i/B_i),\reg G_{i-1}+1\}.
\] 
Combining with \eqref{ineq_G/B}, finally
\[
\reg H_i-i\le \max\{\reg(G_i/B_i)-i,\reg G_{i-1}-(i-1)\} \le \max_{j\ge i-1}\{\reg G_j-j\}\le \reg G.
\] 
The proof is now complete.
\end{proof}
The consequence of Theorem \ref{thm_dim1} which is really needed for the proof of Theorem \ref{thm_cmd1}, is the following statement. The second part of it generalizes the second part of Theorem \ref{thm_finite}.
\begin{cor}
\label{cor_Koszul}
Let $M\in \dgrcatb R$ be a bounded complex. Let $\bsy=y_1,\ldots,y_n \in \mm$ be a sequence of elements of degree $1$. Then
\[
\reg M \le \sup_{i\in \Z} \{\reg H_i(\bsy,M)-i\}.
\]
The equality happens if moreover $\dim H_i(\bsy,M)\le 1$ for all $i\ge 1$.
\end{cor}
We begin with a result on regularity of the Koszul complex.
\begin{lem}
\label{lem_reg_Koszul}
Let $x\in R$ be a homogeneous form of degree $d\ge 1$ and $G\in \dgrcatb R$ a bounded complex. Then
\[
\reg K[x;G]=\reg G+d-1.
\]
\end{lem}
\begin{proof}
Let $F$ be a minimal free resolution of $G$, then $K[x;F]=K[x;R]\otimes_R F$ is a minimal free resolution of $K[x;G]$. Now $K[x;F]_n=F_{n-1}(-d)\oplus F_n$, hence 
\[
\reg K[x;G]=\sup_n \max\{\reg F_{n-1}+d-n,\reg F_n-n\}=\reg G+d-1,
\]
as claimed.
\end{proof}
\begin{proof}[Proof of Corollary \ref{cor_Koszul}]
By repeatedly using Lemma \ref{lem_reg_Koszul}, we see that
\[
\reg  M=\reg  K[\bsy;M].
\]
Using Lemma \ref{exreg_homology}, we get
\[
\reg K[\bsy;M]\le \sup_{i\in \Z} \{\reg H_i(\bsy,M)-i\}.
\]
Hence the first statement follows. The second one is a direct consequence of Theorem \ref{thm_dim1}.
\end{proof}

For the proof of Theorem \ref{thm_cmd1}, first we need a technical notion.
\begin{defn}
Let $M$ be a finitely generated $R$-module. A sequence of linear forms $\bsy=y_1,\ldots,y_n$ is called {\em saturated} with respect to $M$ if the following two conditions are satisfied: 
\begin{enumerate}
\item $y_1,\ldots,y_r$ is an $M$-regular sequence where $r=\min\{n,\depth M\}$, and, 
\item the Koszul homology $H_{>0}(\bsy,M)$ has finite length.
\end{enumerate} 
\end{defn}
Recall that if $M$ is a finitely generated $R$-module, then a form $x\in R_d$ is called filter-regular for $M$ (see \cite{Tr}) if the multiplication $M_{i-d}\xrightarrow{\cdot x}M_i$ is injective for $i\gg 0$. An equivalent condition is that $x$ is a regular element with respect to $M/H^0_{\mm}(M)$. If $x$ is filter-regular for $M$ then $0:_M x\subseteq H^0_{\mm}(M)$, hence it has finite length. The following lemma will be useful.
\begin{lem}
\label{lem_saturated}
Assume that $k$ is an infinite field. Let $M,M_1,\ldots,M_d$ be finitely generated $R$-modules (where $d\ge 1$) and $n=\dim M$. Then there exists a system of parameters consisting of $n$ linear forms of $M$ which is simultaneously saturated with respect to $M_1,\ldots,M_d$.
\end{lem}
We need the following fact, which seems to be folklore, but for completeness, we will include a proof.
\begin{lem}
\label{lem_param}
Let $R$ be a standard graded $k$-algebra, $M$ a finitely generated $R$-module. If $x$ is a parameter of $R/\Ann(M)$ then it is also a parameter of $M$.
\end{lem}
\begin{proof}
Notice that $(x)+\Ann(M) \subseteq \Ann(M/xM) \subseteq \sqrt{(x)+\Ann(M)}$. Therefore if $x$ is a parameter for $R/\Ann(M)$ then
\begin{align*}
\dim R/\Ann(M)-1=\dim R/((x)+\Ann(M))&=\dim R/\Ann(M/xM)\\
                                     &=\dim M/xM.
\end{align*}
The second equality is due to the fact that $\sqrt{(x)+\Ann(M)}=\sqrt{\Ann(M/xM)}$. In particular $\dim M-1=\dim M/xM$, so $x$ is a parameter for $M$.
\end{proof}
\begin{proof}[Proof of Lemma \ref{lem_saturated}]
Use induction to choose $y_1,\ldots,y_i$ where $i\le n$, starting with $i=-1$. Assume that $0\le i \le n$ and we have chosen $y_1,\ldots,y_{i-1}$. Choose $y_i$ so that it is a parameter of $R/\Ann(M/(y_1,\ldots,y_{i-1})M)$, it is filter-regular with respect to $H_j(y_1,\ldots,y_{i-1},M_r)$ for all $j\le i-1, 1\le r \le d$, regular with respect to $M_r/(y_1,\ldots,y_{i-1})M_r$ as soon as $1\le r\le d$ and $\depth M_r/(y_1,\ldots,y_{i-1})M_r\ge 1$. This is possible by prime avoidance and the condition that $R$ is standard graded and $k$ is infinite. By Lemma \ref{lem_param}, $y_1,\ldots,y_n$ is a system of parameters of $M$. Using the exact sequence of Koszul homology, it it clear that $y_1,\ldots,y_n$ satisfy the required condition.
\end{proof}
The main step in the proof of Theorem \ref{thm_cmd1} is done by
\begin{lem}
\label{lem_cmd1}
Let $L \in \dgrcatb R$ be such that $\depth H_i(L)\ge \dim H_r(L)+i-r$ for all $i<r=\sup L$. Then we have 
$\reg L \ge \reg H_r(L)-r.$ 
\end{lem}
\begin{proof}
It is harmless to assume that $k$ is an infinite field, $R$ is a polynomial ring, and $\sup L=0$. Denote $p=\inf L$. Let $F$ be a minimal free resolution of $L$ such that $F_i=0$ for $i<p$. Replacing $L$ by the complex
\[
0\to \frac{F_0}{\img(F_1\to F_0)} \to F_{-1} \to \ldots \to F_{p+1} \to F_p \to 0,
\]
we can assume that $L$ is a bounded complex
\[
L: 0\to L_0\to L_{-1} \to \ldots \to L_{p+1} \to L_p \to 0,
\]
and $\reg L=\max_{p\le i\le 0}\{\reg L_i-i\}$. 

Denote $B_i=\img(L_{i+1}\to L_i), Z_i=\Ker(L_i\to L_{i-1})$ and $H_i=H_i(L)$. We will show that $\reg H_0\le \reg L.$

If $\dim H_0=0$ then we have $\reg H_0\le \reg L_0 \le \reg L$, so it is enough to consider the case $\dim H_0=n>0$. 

Let $\bsy=y_1,\ldots,y_n$ be a minimal system of parameters consisting of linear forms of $H_0$. Using Lemma \ref{lem_saturated}, we can also choose $\bsy$ to be saturated with respect to $B_i,Z_i,L_i,H_i$ for all $i\le 0$. Define the $\ell$th partial Koszul regularity 
$$
\kreg_{\ell} M=\max\{j-i: H_i(\bsy,M)_j\neq 0 ~\textnormal{and $i\ge \ell$}\}
$$ 
for each finitely generated $R$-module $M$ such that $H_{>0}(\bsy,M)$ has finite length and each $\ell \ge 1$. Strictly speaking, $\kreg_{\ell} M$ depends on the sequence $\bsy$. In any case, we have 
$$
\reg M=\max\{\reg H_0(M), \kreg_1 M\}
$$ by Corollary \ref{cor_Koszul} and $\kreg_i M=-\infty$ for all $i>n-\grade(\bsy, M)$.

For each $i\le 0$, we have exact sequences
\begin{align}
0\to Z_i \to L_i \to B_{i-1}\to 0, \label{seq_1}\\
0\to B_{i-1} \to Z_{i-1} \to H_{i-1} \to 0. \label{seq_2}
\end{align}
For $i=0$, looking at the exact sequence of Koszul homology, we deduce
\begin{align}
\reg H_0 \le \max\{\reg L_0, \kreg_1 B_{-1}+1\}, \label{ineq1}\\
\kreg_1 B_{-1}\le \max\{\kreg_1 Z_{-1}, \kreg_2 H_{-1}+1\} \label{ineq2}.
\end{align}
Here is how to get the first inequality. Observe that if $U\to V\to W$ is an exact sequence of finitely generated $R$-modules where $U$ and $V$ have finite length then 
$$
\reg V\le \max\{\reg U,\reg W \}.
$$ Now from the exact sequence
\[
H_1(\bsy,B_{-1})\to H_0(\bsy,H_0)\to H_0(\bsy,L_0)
\]
and the fact that $H_0(\bsy,H_0(L))$ has finite length, we conclude that 
$$
\reg H_0(\bsy,H_0)\le \max\{\reg H_0(\bsy,L_0), \reg H_1(\bsy,B_{-1})\}.
$$
For each $i>0$, similarly, we have
$$
\reg H_i(\bsy,H_0)\le \max\{\reg H_i(\bsy,L_0),\reg H_{i+1}(\bsy,B_{-1})\}.
$$
Therefore the first inequality follows. The second one is obtained by the same trick.

From \eqref{ineq1} and \eqref{ineq2}, we obtain
\[
\reg H_0\le \max\{\reg L_0,\kreg_1 Z_{-1}+1, \kreg_2 H_{-1}+2\}.
\]
Since $\bsy$ is saturated w.r.t.~$H_{-1}$, it holds that $n-\grade(\bsy,H_{-1}) \le \max\{0,n-\depth H_{-1}\}\le 1$. The last inequality follows from $\depth H_{-1} \ge \dim H_0-1$. In particular $H_j(\bsy,H_{-1})=0$ for $j\ge 2$. Hence $\kreg_2 H_{-1}=-\infty$ and thus
\[
\reg H_0\le \max\{\reg L_0, \kreg_1 Z_{-1}+1\}.
\]
For each $i\le 0$, let $L_{\le i}$ be the truncated complex
\[
0\to L_i \to L_{i-1} \to \ldots \to L_p \to 0,
\]
where $L_i$ is in homological position $0$. Then $Z_i=H_0(\shift^{-i}L_{\le i})$, so the above inequality reads 
$$
\reg H_0 \le \max\{\reg L_0, \kreg_1 H_0(\shift^{1}L_{\le -1})+1\}.
$$ 
Arguing with ``partial Koszul regularity" as above, we obtain
\[
\kreg_1 H_0(\shift^{1}L_{\le -1})\le \max\{\kreg_1 L_{-1},\kreg_2 H_0(\shift^{2}L_{\le -2})+1, \kreg_3 H_{-2}+2\}.
\]
The condition $\depth H_{-2}\ge \dim H_0-2$ and the saturatedness of $\bsy$ w.r.t.~ $H_{-2}$ guarantee that $H_j(\bsy, H_{-2})=0$ for $j\ge 3$, so $\kreg_3 H_{-2}=-\infty$. In particular,
\[
\kreg_1 H_0(\shift^{1}L_{\le -1})\le \max\{\kreg_1 L_{-1},\kreg_2 H_0(\shift^{2}L_{\le -2})+1\},
\]
and hence
$$
\reg H_0 \le \max\{\reg L_0, \kreg_1 L_{-1}+1, \kreg_2 H_0(\shift^{2}L_{\le -2})+2\}.
$$
Continuing in this manner, finally 
\begin{align*}
\reg H_0 &\le \max\{\reg L_0,\kreg_1 L_{-1}+1,\ldots,\kreg_n L_{-n}+n\}\\
         &\le \reg L,
\end{align*}
where the first inequality holds since $\kreg_{n+1}H_0(\shift^{n+1}L_{\le -n-1})=-\infty$. This finishes the proof.
\end{proof}
\begin{proof}[Proof of Theorem \ref{thm_cmd1}]
It is harmless to assume that $k$ is an infinite field, $R$ is a polynomial ring and $\sup L=0$. 

Denote $p=\inf L$. We still make the same additional assumptions on $L$ and use the same notations as in the proof of \ref{lem_cmd1}.

By Lemma \ref{exreg_homology}, we get $\reg L\le \max_{p\le i\le 0}\{\reg H_i-i\}$. Therefore it is enough to show that $\reg H_i-i\le \reg L$ for $p\le i\le 0$. We prove by reverse induction on $i\le 0$ the following inequalities
\begin{enumerate}
\item $\reg B_i-i\le \reg L+1$,
\item $\reg L_i/B_i -i \le \reg L,$
\item $\reg H_i-i\le \reg L.$
\end{enumerate} 
Applying Lemma \ref{lem_cmd1} for $L$, we get $\reg H_0\le \reg L$. Also $B_0=0$, hence the conclusion is true for $i=0$.

Take $i<0$. From the exact sequence
\begin{align*}
0\to H_i \to L_i/B_i \to B_{i-1}\to 0,
\end{align*}
we obtain
\begin{align*}
\reg B_{i-1}-(i-1) &\le \max\{\reg L_i/B_i-i+1,\reg H_i-i\}\\
         &\le \reg L+1,
\end{align*}
by induction hypothesis. From the exact sequence 
$$
0\to B_{i-1} \to L_{i-1} \to L_{i-1}/B_{i-1}\to 0,
$$ 
we get
\[
\reg L_{i-1}/B_{i-1}-(i-1)\le \max\{\reg L_{i-1}-(i-1),\reg B_{i-1}-(i-1)-1\} \le \reg L.
\]
Looking at the complex
\[
L': 0\to \frac{L_{i-1}}{B_{i-1}} \to L_{i-2} \to \ldots,
\]
then since $H_j(L')=H_j(L)$ for all $j\le i-1$, by the hypothesis, $L'$ also satisfies the condition of Lemma \ref{lem_cmd1}. Hence 
\begin{align*}
\reg H_{i-1}(L')-(i-1)&=\reg H_{i-1}-(i-1)  \\
                      &\le \reg L'          \\
                      &\le \max\{\reg L_{i-1}/B_{i-1}-(i-1),\reg L_{i-2}-(i-2),\ldots\}.
\end{align*}
The third inequality is a consequence of Proposition \ref{prop_degreewise}. Hence $\reg H_{i-1}-(i-1) \le \reg L$, finishing the induction and the proof.
\end{proof}

\section{Applications}
\label{sect_appl}

As a corollary of the results from Section \ref{sect_bounding_Tor}, we obtain (an important case of) a result originally proved by Chardin using different methods. Notice that this result strengthen Brodmann, Linh and Seiler's Proposition 5.8 in \cite{BLS}. For complexes $M_1,\ldots,M_d$ of $R$-modules (where $d\ge 2$) and $i\in \Z$, denote $\Tor^R_i(M_1,\ldots,M_d)=H_i(F_1\otimes_R \cdots \otimes_R F_d)$ where $F_j$ is a (minimal) free resolution of $M_j$ for each $j$.  
\begin{cor}[Chardin {\cite[Theorem 5.7]{Cha}}]
\label{cor_tensor}
Let $d\ge 2$ and $M_1,\ldots,M_d$ be finitely generated $R$-modules such that at least $d-1$ of them have finite projective dimension. If $\dim \Tor^R_i(M_1,\ldots, M_d)\le 1$ for all $i\ge 1$ then
\[
\max_{i\in \Z}\{\reg \Tor^R_i(M_1,\ldots,M_d)-i\}=\reg M_1+\cdots+\reg M_d-(d-1)\reg R.
\]
\end{cor}
\begin{proof}
We can assume that $M_2,\ldots,M_d$ have finite projective dimension. Consider the complex $G=M_1\dtensor{R}M_2 \dtensor{R}\cdots \dtensor{R}M_d \in \dgrcatb R$. Using Auslander-Buchsbaum formula \ref{AB_exreg} and induction, we have $\reg G=\reg M_1 +\reg M_2+\cdots+\reg M_d -(d-1)\reg R$. The desired inequality follows from Theorem \ref{thm_dim1} since $\Tor^R_i(M_1,\ldots,M_d) =H_i(G)$ for all $i$.
\end{proof}
Now we deduce the results advertised in the introduction.
\begin{proof}[Proof of Theorem \ref{thm_tensor}]
This is a direct consequence of Corollary \ref{cor_tensor}.
\end{proof}
\begin{proof}[Proof of Corollary \ref{cor_Homregularity}]
We apply Lemma \ref{lem_cmd1} for the complex $\RHom R(R/I,N)$. If $\Hom_R(R/I,N)=0$ there is nothing to do.  Otherwise, since $\sup \RHom R(R/I,N)=0$ and $\dim \Hom_R(R/I,N)\le 1$, Lemma \ref{lem_cmd1} tells us that
\[
\reg \Hom_R(R/I,N)\le \reg \RHom{R}(R/I,N)=\reg N.
\] 
The last equality follows from Corollary \ref{cor_reg_Hom}.
\end{proof}
\begin{cor}
\label{cor_tor_Koszul}
Let $R$ be a Koszul algebra, $M, N$ be finitely generated $R$-modules such that $\min\{\projdim_R M, \projdim_R N\}<\infty$. If $\dim \Tor^R_i(M,N)\le 1$ for all $i\ge 1$ then we have the following inequalities for relative regularity:
\[
\reg_R M+\reg_R N-\reg R \le \max_{i\ge 0}\{\reg_R \Tor^R_i(M,N)-i\} \le \reg_R M+\reg_R N+\reg R.
\]
If moreover, both $M$ and $N$ have finite projective dimension, then the lower bound can be improved to $\reg_R M+\reg_R N$.
\end{cor}
\begin{proof}
We can assume that $\projdim_R M<\infty$. By Corollary \ref{AB_reg}, we have $\reg M=\reg_R M+\reg R$. Similar identity holds if $N$ has finite projective dimension. In general, using Theorem \ref{comparison_reg}, we get $\reg_R N \le \reg N \le \reg_R N+\reg R$ and  similar inequalities for $\Tor^R_i(M,N)$. Combining with Corollary \ref{cor_tensor} we get the desired statements.
\end{proof}

\begin{rem}
In general, all inequalities in \ref{cor_tor_Koszul} are strict as illustrated by the following examples. The first is taken from \cite[Remark 2.7]{Con}: let $R=k[x,y]/(x^2+y^2), M=R/(x), N=R/(y)$. Then $\projdim_R M=\projdim_R N=1$ and $\reg_R M=\reg_R N=0$. We see that $\Tor^R_0(M,N)\cong k$ and $\Tor^R_1(M,N)\cong k(-2)$. Hence 
\[
\max_{i\ge 0}\{\reg_R \Tor^R_i(M,N)-i\}=1 >\reg_R M+\reg_R N=0.
\]  
On the other hand, if we let $R$ be any Koszul algebra which is not a polynomial ring, $M=k$ and $N$ be any module, then 
$$
\max_{i\ge 0}\{\reg_R \Tor^R_i(k,N)-i\}=\reg_R N < \reg_R k+\reg_R N+\reg R=\reg_R N+\reg R.
$$
\end{rem}
As another application, we have the following result for regularity of $\Ext$.
\begin{cor}
\label{cor_hom}
Let $M, N$ be finitely generated $R$-modules such that $\RHom R(M,N) \in \dgrcatb R$ \textup{(}e.g., $\min\{\projdim_R M,\injdim_R N\}<\infty$\textup{)}. If for all $i> -\sup \RHom R(M,N)$, it holds that $\depth \Ext^i_R(M,N)\ge \dim \Ext^{i-1}_R(M,N)-1$, then
\[
\max_{i\ge 0}\{\reg \Ext^i_R(M,N)+i\}= \reg N-\indeg M.
\]
\end{cor}
\begin{proof}
Let $L=\RHom R(M,N)\in \dgrcatb R$. From Lemma \ref{cor_reg_Hom}, we have $\reg L=\reg N - \indeg M$. The equality in question is immediate from Theorem \ref{thm_cmd1}.
\end{proof}
\begin{rem}
\label{rem_ChD}
(i) If $\dim (M\otimes_R N)\le 1$ and $\RHom R(M,N)\in \dgrcatb R$ then
\begin{equation}
\label{eq_Ext}
\sup_{i\ge 0}\{\reg \Ext^i_R(M,N)+i\}= \reg N-\indeg M.
\end{equation}
Indeed, we also have $\dim \Ext^i_R(M,N)\le 1$ for all $i\ge 0$, hence the statement follows from Corollary \ref{cor_hom}. This is a result due to Chardin and Divaani-Aazar \cite[Theorem 4.6(1)]{ChD}, where it is assumed that $R$ is a polynomial ring. As pointed out by an anonymous referee, using Theorem 2.5(5) and the second spectral sequence in Lemma 2.3 in \cite{ChD}, one also obtains \eqref{eq_Ext} without the assumptions that $R$ is a polynomial ring and $\RHom R(M,N)$ has bounded homology.

(ii) Corollary \ref{cor_hom} extends \cite[Theorem 3.10]{Ca} and more generally also \cite[Theorem 4.1(ii)]{ChD}. Indeed, the hypothesis is that for $n=\dim R$ and some integer $i_0$, we have $\Ext^i_R(M,N)=0$ for $i<i_0$, $\dim \Ext^{i_0}_R(M,N)\le n-i_0$ and $\Ext^i_R(M,N)$ is either Cohen-Macaulay of dimension $n-i$ or $0$ for $i>i_0$. This means that $\Ext^i_R(M,N)=0$ also for $i>n$, so $\RHom R(M,N)$ is bounded. The remaining conditions of \ref{cor_hom} are clearly fulfilled. 

Finally, Remark \ref{rem_compare} below shows that Corollary \ref{cor_hom} is indeed stronger than \cite[Theorem 4.1(ii)]{ChD}. 
\end{rem}
We also have the following result; the case $0:_M x$ has finite length, e.g. $x$ is $M$-filter-regular, is classical (see \cite[Proposition 1.2]{CH}).
\begin{cor}
\label{cor_filter_regular}
Let $M$ be a finitely generated $R$-module and $x\in R$ a form of degree $d$. If $\depth M/xM\ge \dim (0:_M x)-1$ then 
\begin{equation*}
\reg M=\max\{\reg (0:_M x), \reg M/xM-d+1\}.
\end{equation*}
\end{cor}
\begin{proof}[First proof of Corollary \ref{cor_filter_regular}]
Applying Theorem \ref{thm_cmd1} to the Koszul complex $K[x;M]$, we get
\[
\reg M+d-1=\reg K[x;M]=\max\{\reg H_1(x,M)-1,\reg H_0(x,M)\},
\] 
where the first equality is due to Lemma \ref{lem_reg_Koszul}. Moreover $H_1(x,M)=(0:_M x)(-d)$ and $H_0(x,M)=M/xM$, so we conclude the proof.
\end{proof}
Note that Corollary \ref{cor_filter_regular} is also a special case of Corollary \ref{cor_hom}.
\begin{proof}[Second proof of Corollary \ref{cor_filter_regular}]
Replacing $R$ by a polynomial ring $S$ surjecting onto $R$ and $x$ by a form $y\in S_d$ mapping to $x$, we can assume that $R$ is a polynomial ring. It is harmless to assume that $x\neq 0$. Then $0\to R(-d)\to R\to 0$ is the minimal free resolution of $R/(x)$ over $R$.

Applying Corollary \ref{cor_hom} for the $R$-modules $R/(x)$ and $M$, we have the desired equality. Indeed, $\Ext^1_R(R/(x), M)\cong (M/xM)(d)$ and $\Ext^0_R(R/(x),M)\cong 0:_M x$, and for $i\neq 0,1$, $\Ext^i_R(R/(x),M)=0$. Hence the condition of \ref{cor_hom} is satisfied and 
\[
\reg M-\indeg (R/(x))=\max\{\reg\Ext^1_R(R/(x), M)+1, \reg \Ext^0_R(R/(x),M)\},
\]
or equivalently,
\[
\reg M=\max\{\reg M/xM-d+1,\reg (0:_M x)\},
\]
as desired.
\end{proof}
\begin{ex}
\label{ex_dim2}
Here is an example where one needs the full strength of Corollary \ref{cor_filter_regular}, namely an example with $\dim (0:_M x)\ge 2$. Let $R$ be the determinantal ring of the $2$-minors of
\[
\left(\begin{matrix}
0 & x & z \\
x & y & t
\end{matrix}\right).
\]
Concretely, $R=k[x,y,z,t]/(x^2,xz,xt-yz)$. Now $R/(z)\cong k[x,y,t]/(x^2,xt)$ hence $\depth R/(z)=1$. Moreover, $0:_R z=(\bar{x})\cong [R/(x,z)](-1)$, hence $\dim (0:_R z)=2$. Therefore we can apply Corollary \ref{cor_filter_regular} to conclude that 
\[
\reg R=\max\{\reg R/(z),\reg (0:_R z)\}=\max\{1,1\}=1.
\]
\end{ex}
\begin{rem}
\label{rem_compare}
(i) Our result \ref{cor_hom} is really stronger than Chardin-Divaani-Aazar's \cite[Theorem 4.1(ii)]{ChD}. Indeed, consider the ring $S=k[x,y,z,t]$ and the $S$-modules $S/(z)$ and $R=S/(x^2,xz,xt-yz)$ as in Example \ref{ex_dim2}. The condition of {\it loc.~ cit.} is not satisfied: $\Ext^1_S(S/(z),R)\cong (R/zR)(1)$ has dimension 2 and depth 1, hence it is neither Cohen-Macaulay nor the zero module. On the other hand, $\Ext^0_S(S/(z),R)\cong 0:_R z$ has dimension 2, and the condition of \ref{cor_hom} is satisfied. Therefore, we can apply \ref{cor_hom} to get 
\[
\reg R=\max\{\reg (R/zR), \reg (0:_R z)\}=1,
\]
which was seen in Example \ref{ex_dim2}.

(ii) The presented example simultaneously shows that Theorem \ref{thm_cmd1} is stronger than Theorem \ref{thm_tensor} of Chardin. Indeed, clearly $\Tor^S_1(S/(z),R)\cong (0:_R z)(-1)$ and $\Tor^S_0(S/(z),R)=R/zR$, and $\Tor^S_i(S/(z),R)=0$ for $i\neq 0,1$. Therefore $\dim \Tor^S_1(S/(z),R)=2>1$ and the condition of Theorem \ref{thm_tensor} is not fulfilled. But we still can apply Theorem \ref{thm_cmd1} for the Koszul complex $K[z;R]$ to get 
\[
\reg R=\max\{\reg (R/zR), \reg (0:_R z)\}=1,
\]
since $\depth \Tor^S_0(S/(z),R)=1\ge \dim \Tor^S_1(S/(z),R)-1$.
\end{rem}
\begin{rem}
\label{rem_linear}
Note that in general, the equality in Corollary \ref{cor_filter_regular} is not true if $\depth M/xM\le \dim(0:_M x)-2$. Namely, let $R_n=k[x_1,\ldots,x_n,x_{n+1},y_1,y_2]$ be a polynomial ring and $I_n$ be the ideal of 2-minors of the following matrix
\[
X=\left(\begin{matrix}
0   & x_1 & x_2 & \ldots & x_n & y_1 \\
x_1 & x_2 & x_3 & \ldots & x_{n+1}   & y_2
\end{matrix}\right).
\]
Denote $z=x_{n+1}$. By \cite[Section 5]{C} and \cite[Theorem 4.2]{ZZ}, we have $\reg I_n=2$ and $\reg (I_n+(z))=n+1$, since $I_n+(z)$ is the ideal of $2$-minors of the concatenation matrix with one nilpotent block of size $n+1$ and one scroll block of size $1$. Hence $\reg I_n <\reg (I_n+(z))$ for all $n\ge 2$. For $n=2$, denote $M=T=R/I$ (where $R=R_2,I=I_2$) and $z=x_3$, we have $\depth M/zM=0$ while $\dim(0:_M z)=\dim (I:z)/I=2$. In more details, $(I:z)/I=(y_1^3)T$, hence its annihilator $\Ann_T ((I:z)/I)=(z,x_1,x_2)$, therefore $\dim (I:z)/I=\dim T/(z,x_1,x_2)=\dim k[y_1,y_2]=2$. Hence $\depth M/zM=\dim (0:_M z)-2$ and $\reg M<
\max\{\reg M/zM,\reg (0:_M z)\}$. 

Note that letting $n$ to go to infinity, one also answers in the negative a question due to Caviglia in \cite[Problem 3.7]{PS}, where it was asked whether $\reg (I+(f))$ is bounded by a quadratic function of $\reg I$ for any linear form $f$.
\end{rem}

\section*{Acknowledgements}
We are grateful to Thanh Vu for his useful comments on a previous draft of this paper. We would like to thank an anonymous referee for pointing out a superfluous statement from a previous version, and calling our attention to a paper of Chardin \cite{Cha}. Moreover, we are grateful to them for carefully explaining the relationship between our results and that of \cite{ChD}.

\end{document}